\documentclass{amsart}
\usepackage[utf8]{inputenc}

\title{Non-expansive matrix number systems with bases similar to $J_n(1)$}
\author{Joshua W. Caldwell}
\email{joshua.caldwell@uwaterloo.ca}
\thanks{Research of J. Caldwell was supported in part by the University of Waterloo}
\address{Department of Pure Mathematics, University of Waterloo, Waterloo, Ontario, Canada N2L 3G1}
\author{Kevin G. Hare}
\email{kghare@uwaterloo.ca}
\address{Department of Pure Mathematics, University of Waterloo, Waterloo, Ontario, Canada N2L 3G1}
\author{Tom\'a\v s V\'avra}
\email{tvavra@uwaterloo.ca}
\address{Department of Pure Mathematics, University of Waterloo, Waterloo, Ontario, Canada N2L 3G1}
\thanks{Research of T. V\'avra and K.G. Hare was supported in part by NSERC Grant 2019-03930}

\usepackage{enumitem}
\usepackage{graphicx}
\usepackage{amssymb,latexsym}
\usepackage{amsmath}
\usepackage{amsthm}
\usepackage{amssymb}
\usepackage{color, soul}
\usepackage{hyperref}
\usepackage[dvipsnames]{xcolor}
\usepackage{comment}
\usepackage{tikz-cd}
\usepackage{multicol}
\usepackage{enumitem}
\usepackage{todonotes}

\usepackage{color}
\newtheorem{thm}{Theorem}[section]
\newtheorem{lemma}[thm]{Lemma}
\newtheorem{prop}[thm]{Proposition}
\newtheorem{example}[thm]{Example}
\newtheorem{conj}[thm]{Conjecture}
\newtheorem{remark}[thm]{Remark}
\newtheorem{corollary}[thm]{Corollary}
\newtheorem{definition}[thm]{Definition}

\def\Z{\mathbb Z}
\def\R{\mathbb R}

\def\N{\mathbb N}

\newcommand{\Jtwo}{\begin{pmatrix}1 & 1 \\ 0 & 1 \end{pmatrix}}

\newcommand{\vtwopc}{(0,  1)^T}

\newcommand{\vtwomc}{(0 , -1)^T}

\newcommand{\vthreepc}{(0, 0, 1)^T}

\newcommand{\vthreemc}{(0, 0, -1)^T}

\newcommand{\vfourpc}{(0, 0, 0, 1)^T}

\newcommand{\D}{\mathcal{D}}

\makeatletter
\renewcommand*\env@matrix[1][*\c@MaxMatrixCols c]{%
   \hskip -\arraycolsep
   \let\@ifnextchar\new@ifnextchar
   \array{#1}}
\makeatother

\newsavebox{\mybox}
\begin{lrbox}{\mybox}
$ \left[ \begin{array}{rr}1 & k \\
0 & 1 \end{array} \right]$
\end{lrbox}

\newsavebox{\m}
\begin{lrbox}{\m}
$ \left[ \begin{array}{rr}1 & 1 \\
0 & 1 \end{array} \right]$
\end{lrbox}

\newsavebox{\nm}
\begin{lrbox}{\nm}
$ \left[ \begin{array}{rr}-1 & 1 \\
0 & -1 \end{array} \right]$
\end{lrbox}

\newsavebox{\negativematrix}
\begin{lrbox}{\negativematrix}
$ \left[ \begin{array}{rr}-1 & k \\
0 & -1 \end{array} \right]$
\end{lrbox}

\newsavebox{\matrixone}
\begin{lrbox}{\matrixone}
$ \left( \begin{array}{rr}1 & 1 \\
0 & 1 \end{array} \right)$
\end{lrbox}

\begin{document}
\begin{abstract}
We study representations of integral vectors in a number system with a matrix base $M$ and vector digits. We focus on the case when $M$ is similar to $J_n$, the Jordan block of $1$ of size $n$. If $M=J_2$, we classify digit sets of size 2 allowing representation of the whole $\Z^2$. For $J_n$ with $n\geq 3$, it is shown that three digits suffice to represent all of $\Z^n$. For bases similar to $J_n$, at most $n$ digits are required, with the exception of $n=1$. Moreover, the language of strings representing the zero vector with $M=J_2$ and the digits $(0,\pm 1)^T$ is shown not to be context-free, but to be recognizable by a Turing machine with logarithmic memory.
\end{abstract}
\maketitle




\section{Introduction}

We study representations of integer vectors as combinations $\sum_{i=0}^n M^i a_i$ where the basis $M\in\Z^{n\times n}$ is an integral matrix and the digits $a_i$ take values in a finite set $\mathcal D\subset\Z^n.$ If every element of $\Z^n$ can be represented this way, 
we call the digit set $\mathcal D$ \emph{complete} and
the pair $(M,\mathcal D)$ a \emph{full number system}. 
If every element in $\Z^n$ has a unique representation, 
the pair $(M, \mathcal{D})$ is called a \emph{Canonical number system}.
Such representations can be traced back to the work of Vince~\cite{Vince}, who studied mainly the connection to tessellations of $\R^n.$ It can be easily shown that in order for every point of $\Z^n$ to have a representation with a finite digit set, the basis $M$ must have no eigenvalues strictly inside the unit disc. Vince showed that expanding matrices, i.e. those with all eigenvalues $>1$ in modulus, have a complete digit set. His work has been then expanded to rational matrices and to include also bases with eigenvalues on the unit circle in \cite{JanThu}. If one allows to use also the negative powers of the basis, complete digit sets exist for all non-singular matrices as has been recently shown in~\cite{PeVa}.

Our focus is to study more deeply a relatively simple, but intriguing, case when $M$ is similar to $J_n$, a Jordan block of the number $1$ of dimension $n$. Although such matrices are shown in \cite{JanThu} to have a digits set $\mathcal D$ such that $(M, \mathcal D)$ is a full number system, 
their method is not meant to be constructive, and the digit set one might extract from the proofs is too large. 
In contrast, we give a much tighter result on the size of $\mathcal{D}$ when 
the matrix is $J_n$.
We show that for any dimension $n$, and matrix $M = J_n$, that
there exists an explicit digits set $\D$ of size three such that $(M, \mathcal D)$ is a full number system.
For bases similar to $J_n,$ at most $n$ digits are needed for a full number system.

Let $\D \subset \Z^n$ finite and $M \in \Z^{n\times n}$.
We will say that a vector $v \in \Z^n$ has a representation in the number system
    base $M$ and digit set $\D$ if there exists a finite sequence $a_k, a_{k-1}, \dots, a_0$
    with each $a_{i} \in \D$ where
    \[ v = M^k a_{k} + \dots + M^1 a_{1} + a_{0}. \]
We will denote this by $[a_{k} \dots a_{1} a_{0}]_M$.
When the base is understood we will often denote this by $[a_{k} \dots a_{1} a_{0}]$.

There are many interesting questions that arise when studying generalized number systems.
This includes:
\begin{itemize}
\item Can we represent all integer vectors?
\item Are such representations unique?
\item Is there an algorithm to find a representation?
\item What is the minimal size of a digit set needed to represent all vectors?
\item Which digit sets are complete for a given base?
\end{itemize}

A very common setting is the one where we can represent numbers (or polynomials) uniquely.
Recall, a system $(M, \D)$ is a {\em Canonical number system} if for all vectors $v \in \Z^n$
    there exists a unique representation $v = [a_k a_{k-1} \dots a_0]_M$ with all $a_i \in \D$.
If $M$ is an expansive matrix, and $\D$ chosen such that $(M, \D)$ is a canonical number system,
    then $\D$ contains all residue classes $\bmod\ M$, whence $\#(\D) = |\det(M)|$.
In this case we have the property that
    for all $v \in \Z^n$ there exists a unique $d \in \D$ such that $M^{-1} (v-d) \in \Z^n$.
This uniqueness allows one to construct the unique representation of a vector $v \in \Z^n$ by the backward division algorithm.
Let $v_0 = v$ be the vector we wish to represent.
We find the unique $d \in \D$ such that 
    $M^{-1}(v-d) \in \Z^n$ and we set $a_0$ to be this digit, and $v_1 = M^{-1}(v_0-d)$.
We iteratively construct $v_k = M^{-1}(v_{k-1}-d)$
    in this manner until such time as $v_k \in \D$.
The representation is then $v=[a_0 a_1 a_2 \dots a_k]_M$.
Although the standard approach is to study representations through the transformations that generate them (such as $M^{-1}(v-d)$), we take a different path arising from Lemma~\ref{lem:valid}.

Our setting should be seen in the following context. Canonical number systems require uniqueness, and not all bases admit an appropriate digit set. This theory has been developed in various settings, such as algebraic, polynomial, or matrix bases, see for instance \cite{AKovacs, BKovacs, Petho}. When uniqueness of representations is not required, under certain natural conditions, we can still find a digit set, see \cite{AkThZa} for algebraic bases, or \cite{JanThu} for matrices, however the sizes of the digit sets provided are usually very far from the optimum. Our results lie in the second setting, and we will try to get close to the optimal size for a special class of bases.
    
Our focus is on a class of integer matrices that are not expansive matrices.
In particular, integer matrices such that all eigenvalues $\lambda = 1$.
Particular focus will be given to the $n \times n$ Jordan block with eigenvalue $1$ of the form
\[
J_n :=
\begin{pmatrix}
1 & 1 & 0 & 0& \dots & 0 \\
0 & 1 & 1 & 0& \dots & 0 \\
\vdots & \ddots & \ddots & \ddots& \ddots &\vdots \\
0 & \dots & 0 & 1 & 1 & 0\\
0 & \dots & 0 & 0 & 1 & 1\\
0 & \dots & 0 & 0 & 0 & 1
\end{pmatrix}.
\]
As all eigenvalues of $M$ have modulus $1$, we see that $M^{-1}$ is an integer matrix.

It is worth observing that if both $M$ and $M^{-1}$ are
    integer matrices (as is the case where all eigenvalues are $1$), then $(M, \D)$ cannot be a
    canonical number system, regardless of the choice of $\D$.
To see this, assume that $(M,\D)$ can be used to represent all $v \in \Z^n$.
Consider any non-zero $d \in \D$.
Find a representation such that $-M^{-1} d = [a_k \dots a_0]_M$.
We then see that $[a_n \dots a_0 d] = M (-M^{-1} d) + d = 0,$ the zero vector.
In particular, both $[a_n \dots a_0 d]$ and $[a_n \dots a_0 d a_n \dots a_0 d]$ are representations of the zero vector.
Hence the zero vector has multiple representations.
Further we see that $[a_k \dots a_0 d b_j \dots b_1 b_0]_M = [b_j \dots b_0]_M$
    for all representations $[b_j \dots b_0]_M$.

When there are multiple representations of the same vector, it may be useful to study the language all the strings representing zero. In the case of number systems with a real base, the complexity of the language was linked to the properties of rewriting rules, or to the properties of the so-called spectra of numbers, see \cite{FrPe, Frougny}. Nice properties have been linked to the languages recognizable by finite-state automata. In our case, the language will be more complex: not recognizable by a pushdown automaton, but recognizable by an automaton with memory of logarithmic size in the length of the input.

In Section \ref{sec:Examples} we provide some preliminary results and motivating examples.
We will start our study in Section \ref{sec:J2} by considering digit sets for the base $\Jtwo$. 
In particular, we give a complete characterization of 
    all size two digits sets $\D$ such that 
    $(J_2, \D)$ is a full number system.
This result can be seen as an instance of a generalization of R. Kenyon's open problem \cite{Kenyon} about digit sets for base $3$.
We consider the representations of the zero vector in Section \ref{sec:CFL}. 
It will be shown that language of representations of zero is not a context-free language, but belongs to the class of languages recognizable by a Turing machine with logarithmic memory.
In Section \ref{sec:Jd} we provide a 3 digit set $\D$ for the Jordan block of any dimension such that $(J_n, \D)$ is a full number system.
In Section \ref{sec:similar} we provide a $d$-digit set $\D$ for any matrix $M$ similar to a Jordan block of dimension $d$ such that $(M, \D)$ is a full number system.
Lastly in Section \ref{sec:conc} we make some concluding remarks as well as possible directions for further research.

\section{Preliminary results}
\label{sec:Examples}

In this section we will present some preliminary results and tools that will be used in this study.  We will also provide a number of examples.

An entry of a vector that is arbitrary and its value is not important will be denoted by $*$.
The major strength of the following result is that it allows us to verify a set $(M, \D)$ is an full number system by only verifying a finite number of criteria.

\begin{lemma}
\label{lem:valid}
Let $M := J_n$ be an $n\times n$ Jordan block of the form 
\[ M := J_n :=  \begin{pmatrix} 
1 & 1 & 0 & \dots & 0 \\
0 & 1 & 1 & \dots & 0 \\
\vdots & \ddots & \ddots & \ddots & \vdots \\
0 & \dots & 0 & 1 & 1 \\
0 & \dots & 0 & 0 & 1
\end{pmatrix}
\]
Let $\mathcal{D} = \{d_1, d_2, \dots, d_k\}$.
The set $(M, \D)$ is a full number system if and only if for all $1\leq j\leq n$ there
    exists representations of $4n$ vectors $(*, \dots ,*, A_j, 0 ,\dots, 0)^T$, $(*, \dots ,*, B_j, 0 ,\dots, 0)^T$, $(*, \dots ,*, C_j, 0 ,\dots, 0)^T$, and $(*, \dots ,*, D_j, 0 ,\dots, 0)^T$
    such that
\begin{enumerate}
    \item $A_j > 0 > B_j$, \label{first}
    \item $\gcd(C_j, D_j) = 1$. \label{second}
\end{enumerate}
\end{lemma}

\begin{proof}~
One direction is obvious.

For the other direction, assume that for all $j \leq n$ that such representations exist.
We note that if $[a_k \dots a_0] = (*, * , \dots, a, 0, \dots, 0)^T$ then
$M^r [a_k \dots a_0] = M^r (*,*,\dots,a,0,\dots,0)^T = (*, * , \dots, a,0,\dots,0)^T$ where the choice of $*$ may differ, but the $a$ is fixed.

Using the coprime integers $C_j,D_j$ (not necessarily with opposite signs) we can represent all $n>n_0$ as $n = C_jx_j+D_jy_j$ with $x_j,y_j\in\N.$ Thus $a = C_jx_j+D_jy_j + A_jz_j+B_jw_j$ has a solution for any $a\in\Z$ with $x_j,y_j,z_j,w_j\in\N.$ With the comment above it follows that we can represent any $(*,\dots,*,a,0,\dots,0)^T$ 
by simply appending the representations for \[(*, \dots, *, A_j,0,\dots,0)^T, \dots, (*, \dots, *, D_j,0,\dots,0)^T.\]
Here $a$ is in the $j$th position.

Now we proceed by induction. Assume we are trying to find a representation for $(v_1, \dots, v_n)^T$.
We first find a representation for $(*, \dots, *, *, v_n)^T = (*, \dots, *, v_{n-1}', v_n)^T$ by the approach above.
Here the value of $v_{n-1}'$ can be any value.
We next repeat the process, finding a representation for
    $(*, \dots, *, v_{n-1} - v_{n-1}', 0)^T$.
Appending the previous representation to this representation gives 
    a representation for $(*, \dots, *, v_{n-1}, v_n)^T$.
We repeat this process for all other terms to get the desired result.
\end{proof}

We will denote by $\Z_m$ the integers mod $m$. The following lemma gives us a simple approach for showing that a number system is not full in some cases.
\begin{lemma}
\label{lem:local}
Let $M$ be a $n \times n$ matrix.
If there exists an $m$ such that $(M, \D)$ is not a full number system in $(\Z_m)^n$ then
    $(M, \D)$ is not a full number system  in $\Z^n$.
\end{lemma}

Recall that we say $C \subset \R^n$ is a {\em cone} if 
\begin{itemize}
\item $C$ is a convex set 
\item For all $x \in C$ we have $t x \in C$ for all $t \geq 0$.
\end{itemize}

\begin{lemma}
\label{lem:cone}
Let $M$ be a $n \times n$ matrix and $\D$ a digit set.
If there exists a cone $C \neq \R^n$ such that $M C + d \subset C$ for all $d \in \D$ then 
    $(M, \D)$ is not a full number system.
\end{lemma}

It is worth noting that if $M = J_n$ and for all vectors $(v_1, v_2, \dots, v_n)^T, (w_1, w_2, \dots, w_n)^T \in \mathcal{D}$ we have $w_n v_n \geq 0$ then there will exist a cone $C \neq \R^n$ with the above property.

We will conclude this section by examples of application of the previous results.

\begin{example}
Let $M = J_2$ and $\D = \{d_1, d_2\} = \{\vtwopc, \vtwomc \}.$
\label{example:J2}

We observe that 
\begin{align*}
[d_1] & = (0, 1)^T, \\
[d_2] & = (0, -1)^T, \\
[d_1 d_2] & = (1, 0)^T, \\
[d_2 d_1] & = (-1, 0)^T. \\
\end{align*}

Hence by Lemma \ref{lem:valid} we have that this is a full number system.
\end{example}

\begin{example}\label{example:J3}
Let $M = J_3$ and $\D = \{d_1, d_2\} = \{\vthreepc, \vthreemc \}.$
Notice that $\D \equiv \{\vthreepc, \vthreemc\} \equiv \{(0,0,1)^T\} \bmod 2$.

We will iteratively find all representations $\bmod\ 2$ of length at most $m$, denoted $R_m$, by 
    considering the set of all $M v_1 + v_2 \bmod 2$ where $v_1$ is a representation of length 
    at most $m-1$ and $v_2 \in \D$.

We obtain 
\begin{align*}
R_1 & \equiv \{(0,0,1)^T\} \bmod 2,\\
R_2 & \equiv \{(0,0,1)^T, (0,1,0)^T\} \bmod 2,\\
R_3 & \equiv \{(0,0,1)^T, (0,1,0)^T, (1,1,1)^T\} \bmod 2,\\
R_4 & \equiv \{(0,0,1)^T, (0,1,0)^T, (1,1,1)^T, (0,0,0)^T\} \bmod 2,\\
R_5 & \equiv \{(0,0,1)^T, (0,1,0)^T, (1,1,1)^T, (0,0,0)^T\} \bmod 2.\\
\end{align*}
As $R_4 \equiv R_5 \bmod 2$ we have that $R_m \equiv R_4 \bmod 2$ for all $m \geq 4$.
From this we conclude from Lemma \ref{lem:local} that $(M,D)$ is not a full number system.
\end{example}

\begin{example}
Let $M = J_3$ and  $\D = \{d_1, d_2\} = \{\vthreepc, (0, 1,-2)^T \}$
\label{example:J3w}

We observe that 
\begin{align*}
[d_1] & = (0,0,1)^T, \\
[d_2] & = (0,1,-2)^T, \\
[d_1 d_2 d_1] & = (2,1,0)^T, \\
[d_2 d_1 d_1] & = (0,-2,0)^T, \\
[d_2 d_1 d_1 d_1 d_1 d_2 d_2 d_1 d_1] & = (1, 0, 0)^T, \\
[d_2 d_1 d_1 d_1 d_2 d_1 d_1 d_2 d_1] & = (-5, 0, 0)^T. 
\end{align*}
Hence by Lemma \ref{lem:valid} this is a full number system.
In a similar way we can show that $J_4$ with $\{\vfourpc, (0,0,1,-2)^T\}$ is a full number system where
    the representations for $(A_1, 0, 0, 0)^T$ and $(B_1, 0, 0, 0)^T$ with $A_1 B_1 < 0$ and $\gcd(A_1, B_1) = 1$ are of length 27.
\end{example}

\section{Digit Sets for \usebox{\matrixone}} 
\label{sec:J2}
In this section we will fully classify the size two digit set $\mathcal{D}$ such that 
    $(J_2, \mathcal{D})$ is a full number system.
Let us first describe what those $\mathcal{D}$ such that $(J_2, \mathcal{D})$ is 
    not a full number system.
    
\begin{remark}
\label{rem:gcd}
Let $\D = \{(a,b)^T,(c,d)^T\}$.
From Lemma \ref{lem:local} we see that if $\gcd(b,d) \neq 1$ then $(M, \D)$ is not a full number system.
From Lemma \ref{lem:cone} we see that if $b d \geq 0$ then $(M, \D)$ is not a full number system.
\end{remark}
    
\begin{prop}\label{prop:local_obstruction1}
Let $\D = \{(a,b)^T,(c,d)^T\}$ with $bd<-1.$ If $\gcd(ad-bc,|b|+|d|) \geq 3$, then $\D$ is not a full number system for $J_2.$
\end{prop}

\begin{proof}
Let us put $m = |b|+|d|,$ in particular $b\equiv d\bmod m.$ If $[a_1 a_2\dots a_k] = (X,0)^T,$ for $a_i \in \D$ then we have that $k = n (|b|+|d|)$ for some $n$.
Further we have that there are $n|d|$ of the digits $a_1, \dots, a_k$ are $(a,b)^T$ and $n|b|$ of the digits are $(c,d)^T$.
Let $I_1 = \{i : a_i = (a,b)^T\}$ and $I_2 = \{i: a_i = (c,d)^T\}$.
We have that $X = n(a|d| + c|b|) + \sum_{i\in I_1} bi + \sum_{i\in I_2} di$.
This implies that $X \equiv \pm n(ad-bc) + d\sum_{i=0}^{n(|b|+|d|)-1}i\bmod m$. 
We have that $n(ad-bc)$ is an element of a subgroup of $\Z_m$ with at most $m/2$ elements, and $\sum_{i=0}^{n(|b|+|d|)-1}i$ is either $0\bmod m$ or $(m-1)/2 \bmod m$ depending on $n$ and the parity of $m$. Either way, $X$ does not cover all residual classes in $\Z_m$.
\end{proof}

\begin{prop} \label{prop:mod 4}
Let $\D = \{(a,b)^T,(c,d)^T\}$.
\begin{enumerate}
\item If $b \equiv d \equiv 0 \bmod 2$ then $(J_2, \D)$ is not a full number system. 
\item If $b \equiv d \bmod 4$ and  $b, d$ both odd and $a \equiv c \bmod 2$ then $(J_2, \D)$ is not a full number system.
\item If 
$b \not\equiv d \bmod 4$ and $b, d$ both odd and $a \not\equiv c \bmod 2$ then $(J_2, \D)$ is not a full number system.
\end{enumerate}
\end{prop}

\begin{proof}
It is computable that not all residual classes mod $(\Z_2)^4$ are representable, i.e. the number system is not full by Lemma~\ref{lem:local}.

\end{proof}

\begin{corollary}\label{lem:orbit}
Let $\mathcal{D} = \{(a,b)^T, (c,d)^T\}$.  
If there exists an $n \geq 3$ such that $(a,b)^T \equiv (c,d)^T \bmod{n}$ then 
    $(J_2, \mathcal{D})$ is not a full number system.
\end{corollary}

\begin{proof}
This follows from Lemma \ref{lem:local}  
    as if $a\equiv c, b \equiv d \mod{m}$ for some $m \geq 3$ we have that 
        $m \mid \gcd(ad-bc, |b|+|d|)$, and hence the $\gcd \neq 1$.
\end{proof}

\begin{prop}\label{prop:good}
Let $\D = \{(a,b)^T,(c,d)^T\}$ with $bd=-1.$ Then $(J_2, \D)$ is a full number system if and only if $a\equiv c\bmod 2.$
\end{prop}
\begin{proof}
Let $$M=\begin{pmatrix}1&1\\0&1\end{pmatrix} \quad\text{and}\quad D=\left\{d_1,d_2\right\}=\left\{ (a,1)^T, (b,-1)^T\right\}$$ 
We observe that
\begin{align*}
[v_1] & = (a,1)^T, \\
[v_2] & = (c,-1)^T, \\
[v_1 v_2] & = (a+c+1,0)^T, \\
[v_2 v_1] & = (a+c-1,0)^T.
\end{align*}

Therefore, by Lemma \ref{lem:valid}, we want to find $a,b$ such that $\gcd(a+b+1,a+b-1)=1$. We have $\gcd(a+b+1,a+b-1) = \gcd(2,a+b-1)$. Hence, choosing $a$ and $b$ such that $a \equiv b \bmod 2$ satisfies this requirement. 
We also notice that
\begin{align*}
[\underbrace{v_1 \dots v_1}_{k} \underbrace{v_2 \dots v_2}_{k}] = (k a + k b + k^2,0)^T, \\ 
[\underbrace{v_2 \dots v_2}_{k} \underbrace{v_1 \dots v_1}_{k}] = (k a + k b - k^2,0)^T.
\end{align*}

Hence, we can choose $k$ sufficiently large to get the second part of Lemma \ref{lem:valid}. 
\end{proof}

\begin{thm}
Let $\D = \{(a,b)^T, (c,d)^T\}$ be such that none of the Propositions \ref{prop:local_obstruction1}, \ref{prop:mod 4}, and \ref{prop:good} disprove $(J_2, \D)$ being a full number system. Then $(J_2,\D)$ is a full number system.

\end{thm}

\begin{proof}
We will show that the vectors
\begin{align*}
    (A,0)^T & = [\underbrace{v_1 v_1 \dots v_1}_{|d|} \underbrace{v_2 \dots v_2}_{|b|}], \\
    (B,0)^T & = [\underbrace{v_1 v_1 \dots v_1}_{|d|-1} v_2 v_1  \underbrace{v_2 \dots v_2}_{|b|-1}]
\end{align*}
satisfy $\gcd(A, B) = 1$.

Moreover, if
\begin{align*}
    (C_k,0)^T & = [\underbrace{v_1 v_1 \dots v_1}_{k|d|} \underbrace{v_2 \dots v_2}_{k |b|}], \\
    (D_k,0)^T & = [\underbrace{v_2 v_2 \dots v_2}_{k|b|} \underbrace{v_1 \dots v_1}_{k |d|}],
\end{align*}
then there exists a $k$ sufficiently large such that $C_k D_k < 0$.
Then it follows from Lemma \ref{lem:valid} that $(J_2, \D)$ is a full number system.

Assume without loss of generality that $b > 0 > d$.
We observe that $A = -ad+bc + \frac{b d (b-d)}{2}$ and $A - B = b-d$.
Hence $\gcd(A, B) = \gcd(A, b-d) = \gcd(A,|b| + |d|).$ We will show that this value is either $1,$ or we can apply one of the propositions above.
\begin{enumerate}
    \item If $b\not\equiv d\mod2,$ then $\gcd(A,|b|+|d|) = \gcd(ad-bc,|b|+|d|)\neq 2.$ Hence we either have that $(J_2, \mathcal{D})$ is a full number system, or can apply Proposition~\ref{prop:local_obstruction1}.
    \item Let $b\equiv d\equiv1\mod2.$
    \begin{enumerate}
        \item First assume that $b\equiv d\mod4.$ By Proposition~\ref{prop:mod 4} we may restrict ourselves to the case $a\not\equiv c\mod2.$ In this case $2\nmid ad-bc$ and $2\mid \tfrac{b-d}2,$ so $2\nmid A.$ Thus \[\gcd(A,|b|+|d|)=\gcd(A,\tfrac{|b|+|d|}2)=\gcd(ad-bc,|b|+|d|)\neq2\] and the rest follows as above.
        \item Finally, let $b\not\equiv d\mod 4.$ We can assume that $a\equiv c \mod2$ by Proposition~\ref{prop:mod 4}. Now $2\mid ad-bc,$ $2\nmid bd\tfrac{b-d}2,$ thus $2\nmid\gcd(A,|b|+|d|)=:p.$ We have that $p\mid |b|+|d|$ and $p\mid A,$ and because $2\nmid p,$ also $p\mid \tfrac{|b|+|d|}2.$ Therefore $p \mid \gcd(ad-bc,|b|+|d|)$ and the rest follows as before. 
    \end{enumerate}
\end{enumerate}

Next observe that 
$C_k = \frac{(d-b)bd}{2} k^2 + (bc-ad)k$ and $D_k = \frac{(b-d)bd}{2} k^2 + (bc-ad)k$.
As $d < 0 < b$ we see that $C_k < 0$ for $k$ sufficiently large.
Similarly $D_k > 0$ for $k$ sufficiently large.
\end{proof}

\section{Zero strings}
\label{sec:CFL}
For a good introduction to the theory of computing, see \cite{Sipser12}.
In this section, we explore the representations of $(0,0)^T$ with base $M = J_2$ and digits $\D = \{(0,1)^T, (0,-1)^T\}$. The main result is that the language of all strings representing $(0,0)^T$ is more complex than  context-free languages, but can be recognized by a Turing machine with memory that is logarithmic in the size of the input.

In this section, let $v_1 = (0,1)^T, v_2=(0,-1)^T$ and $L_0=\{v\in\{v_1,v_2\}^+ : [v]_M=(0,0)^T\}$ be the set of non-empty strings of $v_1,v_2$ representing the zero vector.

\begin{thm}
The language $L_0$ is not a context free language.
\end{thm}

\begin{proof}We prove this by contradiction. Suppose $L_0$ is a context free language. Then by the pumping lemma there exists some integer $p$ which is the pumping length of the language $L_0$. Consider the string $w = [v_1^p v_2^{2p} v_1^p] = (0,0)^T$. The pumping lemma tells us that $w$ can be written in the form $s = uvwxy$ where $u,v,w,x,y$ are all substrings such that:
\begin{itemize}
    \item $|vx| \geq 1,$
    \item $|vwx| \leq p,$
    \item $uv^iwx^iy \in L_0$ for every integer $i \geq 0.$
\end{itemize}

Using our choice of $w$ and these facts we see that $vwx$ can only be one of five options:
\begin{enumerate}
    \item $vwx = v_1^j$ for $j \leq p$ \label{c1},
    \item $vwx = v_1^j v_2^k$ for $j+k \leq p$ \label{c2},
    \item $vwx = v_2^j$ for $j \leq p$ \label{c3},
    \item $vwx = v_2^j v_1^k$ for $j+k \leq p$ \label{c4},
    \item $vwx = v_1^j$ for $j \leq p$ \label{c5}.
\end{enumerate}

Consider the word $u v^0 w x^0 y$.  
In cases \eqref{c1}, \eqref{c3} and \eqref{c5} we see that the resulting
    word will have different number of $v_1$ and $v_2$.  
Hence it is not in $L_0$ (and in fact is not in $L$).

In case \eqref{c2} for the word $u v^0 w x^0 y$.
We see that the word is of the form 
    $v_1^{p_1} v_2^{p_2} v_2^p v_1^p$ where $p_1 + p_2 < 2 p$.
If $p_1 \neq p_2$ then this word is not in $L_0$ as it is not in $L$.
If $p_1 = p_2$ then this word evaluates to $(p_1 - p, 0)^T$ where $p_1 < p$.
As such this word is in $L$, but is not in $L_0$.

Case (4) is analogous to (2).
\end{proof}

\begin{thm}
The language $L_0$ belongs to the LSPACE class.
\end{thm}
\begin{proof}
We will show that the elementary algorithm evaluating $[d_1d_{2}\dots d_{n}]$ uses only $\mathcal{O}(\log k)$ space.
The algorithm computes $x_{i+1} = J_2x_{i} + d_{i+1}$ for $i=1,\dots,n-1$ with $x_1 = d_1$. Then $x_k = [d_1d_2\dots d_k].$ The computation of each component of $x_{i+1}$ consists of 3 additions. Working in the binary system, the lengths of inputs into these additions are $\mathcal O(\log k)$ (the worst case scenario being $d_i=v_1$ for all $i$ in which case $x_i = \tfrac{i(i-1)}{2}$).
\end{proof}

\section{Jordan blocks of larger dimensions}
\label{sec:Jd}
We will show that there exists a size three digit set $\mathcal{D}$ such that $(J_n, \mathcal{D})$ is a full number system.
In particular, denote $p = (0,\dots,0,1)^T$, $m = (0,\dots,0,-1)^T$, and $z = (0,\dots,0,0)^T$.
Then the number system $(J_n, \{p,m,z\})$ is a full number system.

Once again we will make use of Lemma \ref{lem:valid}. Define the strings $W_k$ and $Z_k$ by 
\begin{align*}
    W_1 &= pm,\qquad W_k = W_{k-1}E(W_{k-1})\\
    Z_1 &=pm,\qquad Z_k = Z_{k-1} \underbrace{zz\dots z}_{k-1\ \text{times}} E(Z_{k-1})
\end{align*}
where $E$ denotes the morphism swapping $m$ and $p$ while leaving $z$ intact. 
That is $E(w_1 w_2) = E(w_1) E(w_2)$ for any two words $w_1, w_2$, and 
    further $E(p) = m, E(m) = p$ and $E(z) = z$.
    
The first few strings are:
\[
\begin{array}{lll}
W_1 = pm\qquad & W_2 = pmmp\qquad & W_3 = pmmpmppm\\
Z_1 = pm & Z_2 = pmzmp & Z_3 = pmzmpzzmpzpm
\end{array}
\]

One may recognize that $W_k$'s are prefixes of the Thue-Morse sequence. This is no coincidence, for this sequence has been constructed as a solution to the Prouhet–Tarry–Escott problem, see for instance~\cite{AllSha}. It is a problem about partition of the set $\{1,2,\dots,N\}$ into two disjoint sets $S_1,S_2$ such that 
$\sum_{i\in S_1} i^k = \sum_{i\in S_2}i^k$ for all $i=1,2,\dots,k.$ One solution is to define $S_1$ as the indices of letters $p$ and $S_2$ as the indices corresponding to letters $m$. This fact may be found in the proof of the lemma that follows if one expanded the binomial coefficients. 


\begin{lemma}
With $W_k$ described above, we have that $[W_k]_{J_n} = (*,*, \dots, *, w_k, 0, \dots, 0)^T$ where $w_k = 2^{k(k+1)/2}$ occurs in the $(n-k)$th position.
\end{lemma}

\begin{proof}
Let $M=J_n$, we see that 
$$M^k \begin{pmatrix} 0 \\ 0 \\ \vdots \\0 \\ 1\end{pmatrix} 
    = \begin{pmatrix} 1 & k & k\choose 2 & \dots & \dots & \dots & k\choose n-1 \\ 0 & 1 & k & k\choose 2 & \dots & \dots & k\choose n-2 \\ 0 & 0 & 1 & k & k\choose 2 & \dots & k\choose n-3 \\ \vdots & & & & &\ddots &\vdots \\ 0 & 0 & 0 & 0 & 0 & 0 & 1 \end{pmatrix} \begin{pmatrix} 0 \\ 0 \\ 0 \\ \vdots \\ 1 \end{pmatrix} = \begin{pmatrix} {k\choose n-1} \\  {k\choose n-2} \\  {k\choose n-3} \\ \vdots \\ 1 \end{pmatrix}.$$
We observe that the term in the $(n-i)$th position is equal to the coefficient 
    of $y^i$ in $(1+y)^k$.

Consider the first $2^{k}$ terms of the Thue-Morse sequence, given by 
    $t_0, t_1, \dots, t_{2^{k+1}-1} = +1, -1, -1, +1, -1, \dots$.
From \cite{AllSha} we see that $\sum_{\ell=0}^{2^{k}-1} t_\ell x^\ell = \prod_{\ell=0}^{k-1} (1-x^{2^\ell})$.
Let $\mathrm{Rev}(w)$ be the reverse of a word.  That is
    $\mathrm{Rev}(a_1 a_2 \dots a_d) = a_d a_{d-1} \dots a_1$.
This gives us that
\begin{align*} [\mathrm{Rev}(W_k)]_{J_n} &= (-1)^{k}
(t_0 M^0 + t_1 M^1 + \dots + t_{2^{k}-1} M^{2^{k}-1}) p \\ 
 & = \left(\prod_{\ell=0}^{k-1} (I-M^{2^{\ell}})\right) p,
\end{align*}
where $I$ is the identity matrix.
Here the factor of $(-1)^k$ results from the fact that the $\mathrm{Rev}(W_k) = W_k$  when $k$ is even and $E(W_k)$ when $k$ is odd.

We further observer that the term in the $(n-i)$th position of this resulting vector is 
     given by the coefficient of $y^i$ in the product $(-1)^k\prod_{\ell=0}^{k-1} (1-(1+y)^{2^\ell})$.
This product is 
\begin{align*}
(-1)^k \prod_{\ell=0}^{k-1}(1-(1+y)^{2^\ell}) & = (-1)^k\prod_{\ell=0}^{k-1} -\left({2^\ell\choose 1} y + {2^\ell \choose 2} y^2 + \dots \right) \\
    & = \left(\prod_{\ell=0}^{k-1} 2^\ell\right)y^k + \text{higher order terms}\\
    & = 2^{k (k-1)/2} y^k + \text{higher order terms}.
\end{align*}
This proves that the value of the $(n-k)$th term in $W_k$ is $2^{k(k+1)/2}$ as desired.  
Further we have that the value of the $(n-i)$th term in $W_k$ is $0$ for all $i < k$.
\end{proof}

\begin{lemma}
With $Z_k$ described above, we have that $[Z_k]_{J_n} = (*,*, \dots, *, z_k, 0, \dots, 0)^T$ where $z_k = \prod_{\ell=0}^{k-1} (2^{\ell+1}-1)$ occurs in the $(n-\ell)$th position.
\end{lemma}

\begin{proof}
This is done in a similar way, using the observation that the generating function 
    for the sequence $+1, -1, 0, -1, +1, 0, 0, \dots$ is given by 
    $\prod_{\ell=0}^{k-1} (1-x^{2^{\ell+1}-1})$.
\end{proof}

To conclude, the following theorem follows from Lemma \ref{lem:valid} when realizing that $w_n$ and $z_n$ are coprime and $[E(a_k\dots a_0)] = -[a_k\dots a_0]$ for any $a_k\dots a_0\in\{p,m,z\}^{k+1}.$
\begin{thm}
The number system $(J_n,\{p,m,z\})$ is a full number system.
\end{thm}

\section{Matrices similar to $J_n$}
\label{sec:similar}

Let $M$ be similar to a Jordan block of dimension $n$.
In this section we will investigate the minimal size of $\mathcal{D}$ such that 
    $(M, \mathcal{D})$ is a full number system.
For this purpose we are going to need a notion of numeration systems for lattices.
\begin{definition}
Let $M,B\in\Z^{n\times n}$ and $\D\subset B\Z^n$. We say that $(M,\D)$ is a full number system for the lattice $B\Z^n$, if every $z\in B\Z^n$ can be expressed as $z = \sum_{i=0}^k d_i M^i$ for some choice
    of $d_i \in \mathcal{D}$ and some $k$.
\end{definition}

First, we provide a lattice analogy of Lemma~\ref{lem:valid}.

\begin{lemma}
\label{lem:valid lattice}
Let $J_n$ be the $n\times n$ Jordan block and let $\mathcal{L}=B\Z^n$ where $B = (b_{i,j})$ is an upper triangular matrix.

Then  $(M, \D)$ is a full number system for $\mathcal{L}$ if and only if for all $1\leq j\leq n$ there exist representations of
$4n$ vectors  $(*, \dots ,*, A_j, 0 ,\dots, 0)^T$, $(*, \dots ,*, B_j, 0 ,\dots, 0)^T$, $( *, \dots ,*, C_j, 0 ,\dots, 0)^T$ and $ (*, \dots ,*, D_j, 0 ,\dots, 0)^T$ such that

\begin{enumerate}
    \item $A_j > 0 > B_j$
    \item $\gcd(C_j, D_j) = b_{j,j}$
\end{enumerate}
\end{lemma}
If $(*,\dots,a,0,\dots,0)^T\in B\Z^n$ with $B$ upper triangular, then $a = k b_{j,j}$ for some $k\in\Z$, where $j$ is the position of $a$. The rest of the proof is similar to the proof of Lemma \ref{lem:valid}. 

The following lemma is a corollary of the Hermite normal form for integer matrices.
\begin{lemma}
\label{lem:hermite}
Let $P \in \mathbb{Z}^{n \times n}$.
There exists an upper triangular integer matrix $B$ and unimodular matrix $U$ such that $P = B U$.
\end{lemma}

\begin{proof}
Let $c = \det(P)$.
By putting $c P^{-1}$ into Hermite normal form, 
    there exists an upper triangular 
    matrix $H$ and a unimodular matrix $U'$ such that $c U' P^{-1} = H$.
Taking inverses of both sides we have $P U'^{-1} = c H^{-1}$.  
This gives us that $P = c H^{-1} U'$.
Taking $B = c H^{-1}$ we have the desired result. Note that $B$ is integral, because $B = PU'^{-1}$ with $P$ integral and $U'$ unitary.
\end{proof}
Finally, the following theorem gives the size of the smallest digit set $\mathcal{D}$ such that $(M, \mathcal{D})$ is a full number system for bases of dimension $d$ similar to the Jordan block $J_n$.
\begin{thm}
\label{thm:similar}
Let $n \geq 2$ and $M\in\Z^{n\times n}$ be similar to $J_n.$ Then there exists a digit set $\mathcal{D}$ of size $n$ such that $(M, \mathcal{D})$ is a full number system. Moreover, this bound is tight, i.e. for each $n$ there exists $M$ similar to $J_n$ for which any digit set with less then $n$ elements does not give a full number system.
\end{thm}

\begin{proof}
Let $M = P^{-1}J_n P$ with $P$ integral. Then the equality $z = \sum_{i=0}^N M^id_i$ can be rewritten as $Pz = \sum_{i=0}^N J^iPd_i.$ Thus, $(M,\D)$ is a full number system if and only if $(J_n,P\D)$ is a full number system for $P\Z^n$. By Lemma~\ref{lem:hermite}, $\mathcal L = B\Z^n$ with $B\in\Z^{n\times n}$ upper triangular. Moreover, we have that $J\mathcal L\subseteq\mathcal L$ which implies that the diagonal of $B=(b_{i,j})$ has the property $b_i\mid b_{i+1}$ (consider the equality $JBe_i = Bz$ for $z\in\Z^n$).

We will proceed by induction.
Assume first that $n=2,$ i.e. $B = \begin{pmatrix}a & b\\ 0&c\end{pmatrix}$ with $a,c\neq0.$ We will distinguish three cases.

\begin{enumerate}[label={Case {{\arabic*}}:}]
    \item {\bf $c=2ka$}

Take $\mathcal D=\{B(0,1)^T,B(1,-1)^T\}=\{(b,c)^T,(a-b,-c)^T\}.$ Then we observe that base $J_2$ that
\begin{align*}
    [v_1v_2]_{J_2} & = (a+c,0)^T \\
    [v_2v_1]_{J_2} & = (a-c,0)^T
\end{align*}
We can see that $\gcd(a+c,a-c) = a$, we obtain the statement by Lemma \ref{lem:valid}.

\item {\bf $c=(2k+1)a, k \neq 0, -1$} 

Take $\mathcal D=\{B(1,1)^T,B(1,-1)^T\}=\{(a+b,c)^T,(a-b,-c)^T\}.$ We observe that
\begin{align*}
    [v_1v_2]_{J_2} & = (2a+c,0)^T \\
    [v_2v_1]_{J_2} & = (2a-c,0)^T
\end{align*}
The rest is analogous to the previous case.

\item {\bf $c = \pm a$}

We will do the case $c = a$ only.  The case $c = -a$ is similar.
Take $\mathcal D=\{B(1,1)^T,B(1,-1)^T\}=\}\{(a+b,a)^T,(a-b,-a)^T\}.$
We observe that

\begin{align*}
    [v_2v_1]_{J_2} & = (a,0)^T \\
    [v_2 v_2 v_2 v_1 v_1 v_1]_{J_2} & = (-3a,0)^T
\end{align*}
The rest is analogous to the previous cases.
\end{enumerate}

We will now proceed by induction.
Let $\mathcal{L} = B \mathbb{Z}^n$ and $\mathcal{L'} = B' \mathbb{Z}^{n-1}$ where $B'$ is the $(n-1) \times (n-1)$ lower right sub-matrix of $B$.
We know by induction there exists a size $n-1$ digit set $\mathcal{D}' = \{d_1', d_2', \dots, d_{n-1}'\}$ such that $(J_{n-1}, \mathcal{D}')$ is a full number system for the lattice 
$\mathcal{L}'$.
As $d_i' \in \mathcal{L}'$ we see that there exists a vector $v' \in \mathbb{Z}^{n-1}$ such that $d_i' = B' v'$.
Extend each $d_i'$ in $\mathbb{Z}^{n-1}$ to a vector $d_i$ in $\mathbb{Z}^n$ by prepending $0$ to $v'$ and setting $d_i = B v$.
That is, if $d_i' = B (v_2, v_3, \dots, v_n)^T$ let $d_i = B (0, v_2, v_3, \dots, v_n)^T$.
As $\mathcal{D}'$ if a full number system for the lattice $\mathcal{L}'$ there exists a 
    non-trivial representation of $0 \in \mathbb{Z}^{n-1}$, say $[a_1' a_2' \dots a_k']_{J_{n-1}}$ where $a_i' \in \mathcal{D}'$.
Consider the extended word in $\mathbb{Z}^n$ of $[a_1 a_2 \dots a_k]_{J_n} = (A, 0, \dots, 0)^T$.
Assume without loss of generality that the upper left entry of $B$ is positive.
If $A > 0$ then let $d_n = B (-1, 0, \dots, 0)^T$.  In this case $(J_n, \{d_1, \dots, d_n\}$ is a full number system for the lattice $\mathcal{L}.$
If $A < 0$ then let $d_n = B (1, 0, \dots, 0)^T$.  Again, $((J)_n, \{d_1, \dots, d_n\})$ is a full number system for the lattice $\mathcal{L}$.

If $A = 0$ then consider any digit $d_i$ which occurs a non-zero number of times in the 
   expression $[a_1 a_2 \dots a_k]_{J_n}$.
Let $d_i^* = B (1, v_2, v_3, \dots, v_n)^T$ and redefine $\mathcal{D} = \{d_1, d_2, \dots, d_{i-1}, d_i^*, d_{i+1}, \dots, d_n\}$.
We see in this case that the representation $[a_1 \dots a_k]_{J_n} = (A, 0, \dots, 0)^T$ with
    $a_i \in \mathcal{D}$ has $A \neq 0$ and the proof follows as before.

On the other hand, consider the $n \times n$ matrix
\[
M = \begin{pmatrix}
1 & 2 & 0 & \dots & \dots & 0 \\
0 & 1 & 2 & 0 & \dots & 0 \\
\vdots & \ddots & \ddots & \ddots & \ddots & \vdots \\
0 & \dots & 0 & 1 & 2 & 0 \\
0 & \dots & \dots & 0 & 1 & 2 \\
0 & \dots & \dots & \dots & 0 & 1
\end{pmatrix}
\]
It is quick to check that $M$ is similar to a $J_n$. Let $\mathcal{D} = \{v_1, v_2, v_3,\dots,v_k\}$ be any $k$-digit alphabet.
We see that $M \equiv M^r \bmod 2$ for all $r$.
Hence in $\Z_2^n$ that the set of representable numbers
  is restricted to $\{ a_1 v_1 + a_2 v_2 + \dots + a_k v_k: a_j \in \Z_2\}.$
This has size at most $2^k$, whereas the size of $\Z_2^n$ is $2^n$.  Therefore we must have $k\geq n$.

\end{proof}

\begin{example}
Start with 
\[
M = \begin{pmatrix}
-1 & 2 & 0 & 0 \\
-2 & 3 & 0 & 0 \\
-2 & 0 & -1 & 2 \\
-3 & 2 & -2 & 3
\end{pmatrix}
\]
We have $M = P^{-1} J_4 P$ where 
\[
P = \begin{pmatrix}
2 & -1 & 0 & -1 \\
1 & 0 & 2 & -2 \\
0 & -2 & 0 & 0 \\
4 & -4 & 0 & 0
\end{pmatrix}\]
Next we have that $P = BU$ where
\[
B = 
\begin{pmatrix}
1 & 0 & 0 & 1  \\
0 & 2& 1 & 1  \\
0 & 0 & 2 & 0 \\
0 & 0 & 0 & 4
\end{pmatrix}
\text{  and  }
U = \begin{pmatrix}
1 & 0 & 0 & -1 \\
0 & 1 & 1 & -1 \\
0 & -1 &  0 & 0 \\
1 & -1 & 0 & 0
\end{pmatrix}
\]

We will define 
\[
B' = \begin{pmatrix} 2 & 1 & 1 \\ 0 & 2 & 0 \\ 0 & 0 & 4 \end{pmatrix}, \text{  } B'' = \begin{pmatrix} 2 & 0 \\ 0 & 4 \end{pmatrix}, \text{  } \mathcal{L}' = B' \mathbb{Z}^3,
\text{  and  } \mathcal{L}'' = B'' \mathbb{Z}^2
\]
By first considering $B'' = \begin{pmatrix} 2& 0  \\ 0 & 4 \end{pmatrix}$ we see that this follows from the first case.
Take $\mathcal{D}'' = \{v_3'', v_4''\} = \{(0, 4)^T, (2, -4)^T\}$.
We note that $v_3'' = B'' (0,1)^T, v_4'' = B'' (1,-1)^T \in \mathcal{L}''$.
We check that $[v_3'' v_4''] = (6, 0)^T$ and $[v_4'' v_3''] = (-2, 0)^T$.
Hence $(J_2, \mathcal{D}'')$ is a full number system for 
     the lattice $\mathcal{L}''$.
     
We see that a non-trivial representation of $(0,0)^T$ is 
$[v_3'' v_4'' v_4'' v_3'' v_4'' v_3'' v_4'' v_3'']$.

Consider the digit set $\{v_3', v_4'\} = \{(1, 0, 4)^T, (0, 2, -4)^T\}$.
Here we have that $v_3' = B' (0,0,1)^T$ and $v_4 = B' (0,1,-1)^T$.

We see that $[v_3' v_4' v_4' v_3' v_4' v_3' v_4' v_3'] = (34, 0, 0)^T$.
Take $v_2' = (-2, 0, 0)^T = B' (-1, 0, 0)^T$ and set $\mathcal{D}' = \{v_2', v_3', v_4'\}$.
We have that $(J_3, \mathcal{D}')$ is a full number system for the 
    lattice $\mathcal{L}'$.

We see that a non-trivial representation of $(0,0,0)^T$ is 
$[v_3' v_4' v_4' v_3' v_4' v_3' v_4' v_3'
\underbrace{v_2' \dots v_2'}_{17}]$.

Consider the digit set $\{v_2, v_3, v_4\} = \{(0,-2, 0,0)^T, (1,1,0,4)^T,(-1, 0,2, -4)^T\}$.  
Here $v_2 = B (0,-1,0,0)^T, v_3 = B (1,1,0,4)^T, v_4 = B(0,0,1,-1)^T \in \mathcal{L}$.
We see that 
$[v_3 v_4 v_4 v_3 v_4 v_3 v_4 v_3
\underbrace{v_2 \dots v_2}_{17}] = (407, 0, 0, 0)^T$.
Take $v_1 = [-1,0,0,0] = B (-1, 0, 0, 0)^T$ 
and set $\mathcal{D} = \{v_1, v_2, v_3, v_4\}$.
We then see that $(J_4, \mathcal{D})$ is a full number system
    for $\mathcal{L}$.
    
We next denote that 
\begin{align*}
    w_1 & = P^{-1} v_1 = (0,0,1,1)^T \\
    w_2 & = P^{-1} v_2 = (0,0,-1,0)^T \\
    w_3 & = P^{-1} v_3 = (1,0,1,1)^T \\
    w_4 & = P^{-1} v_4 = (-2,-1,-1,-2)^T
\end{align*}
From Theorem \ref{thm:similar} we have that $(M, \{w_1, w_2, w_3, w_4\})$ is a full number system.
\end{example}

\section{Final comments, open question and conjectures}
\label{sec:conc}

For $J_2$ and a digit set $(0,1)^T, (0,-1)^T$ we know that we can represent every value $(a,b)$.
What is not clear is how to find the shortest representation.  
It is known that such a representation is not necessarily unique.  
For example 
\[ \begin{pmatrix} 5 \\ -1 \end{pmatrix} = [v_1 v_2 v_2 v_1 v_2] = [v_2 v_1 v_1 v_2 v_2] \]
are both representations of minimal length.
It would be interesting to find an algorithm that would provide a shortest representation.

In Example \ref{example:J2} we gave a 2-digit alphabet for $J_2$.
In Example \ref{example:J3w} we indicated how to find a 2-digit alphabet for the $J_3$ and $J_4$. 
In Section
    \ref{sec:similar} we showed that there exists a 3-digit alphabet for $J_5$.
It is clear that there does not exist a 1-digit alphabet for $J_5$ by Lemma \ref{lem:cone}.
What is not clear is if the 3-digit alphabet is an alphabet of minimal size for $J_5$.
More generally, does there exist a $2$-digit alphabet for $J_n$ for any $n \geq 2$?

In most of this paper we focused on the case where $M$ was a Jordan block of the form $J_n$ for some $n$.
In Section \ref{sec:similar} we discussed when $M$ was similar to a Jordan block.
It would be interesting to know what happens more generally.  
That is, what if $M$ has a decomposition into multiple non-expansive Jordan blocks.
An obvious digit set would be the cross product of the digits sets for the various blocks.
What is not clear is if the non-uniqueness of expansions would allow for a smaller digit set.

In Section \ref{sec:CFL} we focused on the case of $M = J_2$ with a specific digit set.  It would be interesting to know if this result was true in greater generality.

\bibliographystyle{plain}

\end{document}